\documentclass{amsart}

\usepackage{paralist}
\usepackage{url}
\usepackage{hyperref}
\usepackage{color}
\usepackage{mathrsfs}
\usepackage{graphicx}
\usepackage[T1]{fontenc}
\usepackage{booktabs}
\usepackage{cleveref}
\usepackage{multirow}
\usepackage[figuresright]{rotating}

\newcommand{\RR}{{\mathbb R}}

\newcommand{\re}{\mathbb{R}}

\newcommand{\N}{\mathbb{N}}

\newcommand{\cal}{\mathcal}

\def\af{\alpha}

\newcommand{\ud}{\mathrm{d}}

\newcommand{\td}[1]{\tilde{#1}}

\newcommand{\rank}{\mbox{\upshape rank}}
\newcommand{\sos}{\mbox{\upshape\tiny sos}}
\newcommand{\mom}{\mbox{\upshape\tiny mom}}

\newcommand{\Qm}{\mathcal{Q}}
\newcommand{\Qk}{\mathcal{Q}_k}
\newcommand{\wt}[1]{\widetilde{#1}}
\newcommand{\So}{\wt{S}_>}

\newcommand{\St}{\wt{S}}

\newtheorem{theorem}{Theorem}[section]

\newtheorem{prop}[theorem]{Proposition}
\newtheorem{condition}[theorem]{Condition}

\newtheorem{lemma}[theorem]{Lemma}
\newtheorem{cor}[theorem]{Corollary}

\newtheorem{ass}[theorem]{Assumption}

\newtheorem{define}[theorem]{Definition}
\newtheorem{definition}[theorem]{Definition}
\theoremstyle{definition}
\newtheorem{example}[theorem]{Example}

\newtheorem{remark}[theorem]{Remark}

\newenvironment{examplerv}[1]{\refstepcounter{theorem}
\par\noindent{\bf Example #1 revisited}. \ignorespaces}{}
\numberwithin{equation}{section}

\begin{document}
\title[SDP relaxations for linear semi-infinite polynomial programming]
{Semidefinite programming
relaxations for linear semi-infinite polynomial programming}
\author{Feng Guo}
\address{School of Mathematical Sciences, Dalian  University of Technology,
Dalian, 116024, China}
\email{fguo@dlut.edu.cn}

\author{Xiaoxia Sun}
\address{School of Mathematics,
Dongbei University of Finance and Economics,
Dalian, 116025, China}
\email{xiaoxiasun@dufe.edu.cn}

\begin{abstract}
	This paper studies a class of so-called
	linear semi-infinite polynomial programming (LSIPP) problems.
	It is a subclass of linear semi-infinite programming problems
	whose constraint functions are polynomials in parameters
	and index sets are basic semialgebraic sets.
	We present a hierarchy of semidefinite programming (SDP)
	relaxations for LSIPP problems. Convergence rate analysis of the
	SDP relaxations is established based on some existing results. We
	show how to verify the compactness of feasible sets of LSIPP
	problems. In the end, we extend the SDP relaxation method to more
	general semi-infinite programming problems.
\end{abstract}
\date{}
\maketitle



\section{introduction}
We consider the following
{\itshape linear semi-infinite polynomial programming}
(LSIPP) problem
\begin{equation}\label{eq::lsipp}
(P)\qquad	
\left\{
\begin{aligned}
	p^*:=\inf_{x\in\RR^m}&\ c^Tx\\
	\text{s.t.}&\ a(y)^Tx+b(y)\ge 0,\ \ \forall y\in S\subseteq\RR^n,
\end{aligned}\right.
\end{equation}
where $c\in\RR^m$, $b(Y)\in\RR[Y]:=\RR[Y_1,\ldots,Y_n]$ the
polynomial ring in $Y$ over the real field,
$a(Y)=(a_1(Y),\ldots,a_m(Y))^T\in\RR[Y]^m$, and the {\itshape index set}
$S$ is a basic semialgebraic set defined by
\begin{equation}
	S:=\{y\in\RR^n\mid g_1(y)\ge 0, \ldots, g_s(y)\ge 0\},
	\label{eq::S}
\end{equation}
where $g_j(Y)\in\RR[Y]$, $j=1,\dots,s$.
Lowercase letters (e.g. $x, y, w$) are hereinafter used for denoting
points in a space while uppercase letters (e.g. $X,Y,W$) for the
corresponding variables.
In this paper, we assume that
$(\ref{eq::lsipp})$ is feasible and bounded from below, i.e.,
$-\infty<p^*<\infty$.
Note that
the problem $(\ref{eq::lsipp})$ is NP-hard.
Indeed, it is obvious that the
problem of
minimizing a polynomial $f(Y)\in\RR[Y]$ over $S$ can be regarded
as a special LSIPP problem (see Example~\ref{ex::stable}).
As is well known, the polynomial optimization problem
is NP-hard even when $n>1$, $f(Y)$ is a nonconvex
quadratic polynomial and $g_j(Y)$'s are linear \cite{PardalosVavasis1991}.
Hence, a general LSIPP problem
cannot
be expected to be solved in polynomial time unless P=NP.

LSIPP can be seen as a special branch of
{\itshape linear semi-infinite programming} (LSIP), or more general, of
{\itshape semi-infinite programming} (SIP), in which the involved functions
are not necessarily polynomials.
Numerically, SIP problems can be solved by different approaches
including, for instance, discretization methods,
local reduction methods, exchange methods,
simplex-like methods and so on.
See the surveys
\cite{LSIO,LSIP,Still2007}
and the references therein for details.
One of main difficulties in numerical treatment of general SIP problems
is that the feasibility test of  a point $\bar{u}\in\RR^m$
is equivalent to {\itshape globally} solving the problem of minimizing the
constraint function with fixed $\bar{u}$ over the index set,
which is called the lower level subproblem.
Typically,
when solving SIP problems by existing methods in the literature,
the main difficulty lies in solving
the nonlinear lower level subproblems
at each iteration.

LSIPP, as a special subclass of SIP,
has many applications like minimax problems,
functional approximation problems. However,
to the best of our knowledge, few of the numerical methods mentioned above
are specially designed
by exploiting features of polynomial optimization problems.
Parpas and Rustem \cite{PR2009} proposed a
discretization-like method to solve minimax polynomial optimization
problems, which can be reformulated as
semi-infinite polynomial programming (SIPP)
problems.
Using polynomial approximation and an appropriate hierarchy of
{\itshape semidefinite programming} (SDP) relaxations,
Lasserre presented an
algorithm to solve the generalized SIPP problems in
\cite{Lasserre2011}. Based on an exchange scheme, an SDP
relaxation method for solving
SIPP problems was proposed in \cite{SIPPSDP}.
By using representations of nonnegative polynomials in the
univariate case, an SDP method was given in \cite{XSQ}
for LSIPP problems $(\ref{eq::lsipp})$ with $S$ being closed
intervals.


In this paper, we propose a hierarchy of SDP relaxations for LSIPP
\eqref{eq::lsipp}. 
The dual problem of LSIPP is a special case of the generalized moment
problems (Section~\ref{sec::dual}),  which has been well investigated, see
\cite{CMP1965,AK1962,MPTAT,K1980,MM,Lasserre2008,Nie201501} and the
references therein.
Lasserre \cite{Lasserre2008} proposed an SDP relaxation method for
generalized moment problems based on Putinar's Positivstellensatz
\cite{Putinar1993}. 
Although the SDP relaxations presented in this paper can be seen as 
the dual of Lasserre's relaxations for GPM, 
they are of their independent interest because of the following
desirable features they enjoy. First, some (approximate)
minimizers of \eqref{eq::lsipp} can be extracted by these SDP
relaxations, which is very useful in some applications, like
functional approximation problems (Example~\ref{ex::af}); Second,
convergence rate of these SDP relaxations can be estimated
(Section~\ref{sec::cra}) by using the
complexity analysis of Putinar's Positivstellensatz in
\cite{NieSchweighofer}; Third, these SDP relaxations can be
easily extended to more general semi-infinite programming problems
(Section~\ref{sec::extension}), like problems of the form
\eqref{eq::lsipp} with semialgebraic functions, or with s.o.s-convex
objectives. As the feasible set of \eqref{eq::lsipp} is assumed to be
compact in the convergence rate estimation, we also show that the
compactness can be verified by
computing a positive lower bound of the infima of several LSIPP
problems. It can be done by the proposed SDP relaxations if the
finite convergence happens (in particular, if $S$ is a closed and
bounded interval (Section~\ref{sec::compacttest})). 

\vskip 5pt
This paper is organized as follows.
We introduce some notation and preliminaries
in Section \ref{sec::pre}. 
SDP relaxations of LSIPP problems and the convergence rate analysis is
given in Section \ref{section::SDPrelax}, where we also discuss how to
verify the compactness of feasible sets of LSIPP problems. In
Section~\ref{sec::extension}, we extend the SDP relaxation method to
more general semi-infinite programming problems.
Some conclusions are made in Section \ref{sec::conclusion}.

\section{Notation and Preliminaries}\label{sec::pre}

Here is some notation used in this paper.
The symbol $\N$ (resp., $\re$) denotes
the set of nonnegative integers (resp., real
numbers). For any $t\in \re$, $\lceil t\rceil$
denotes the smallest integer that is not smaller than $t$.
For $y=(y_1,\ldots,y_n)\in \re^n$,
$\Vert y\Vert_2$ denotes the standard Euclidean norm
of $y$.
For $\alpha=(\alpha_1,\ldots,\alpha_n)\in\N^n$,
$\Vert\alpha\Vert_1:=\alpha_1+\cdots+\alpha_n$.
For $k\in\N$, denote $\N^n_k=\{\alpha\in\N^n\mid\Vert\alpha\Vert_1\le k\}$.
For $y \in \re^n$ and $\af \in \N^n$, $y^\af$ denotes
$y_1^{\af_1}\cdots y_n^{\af_n}$.
$\mathbb{R}[Y] = \mathbb{R}[Y_1,\cdots,Y_n]$ denotes
the ring of polynomials in $(Y_1,\cdots,Y_n)$ with real
coefficients. For $k\in\N$, denote by $\RR[Y]_k$ the set of polynomials
in $\RR[Y]$ of total degree up to $k$.
For a symmetric matrix $Z$, $Z\succeq 0 (\succ
0)$ means that $Z$ is positive semidefinite (definite).
For $m\in\N$, $\RR^{m\times m}$ denotes the set of $m\times m$ real
matrices and $\mathbb{S}_+^m\subset\RR^{m\times m}$ denotes its subset of positive
semidefinite matrices. 
For two symmetric matrices $A, B$ of the same size,
$\langle A, B\rangle$ denotes the inner product of $A$
and $B$.

\subsection{Sums of squares and moments}
We recall some background about
{\itshape sums of squares} (s.o.s) of polynomials
and the dual theory of {\itshape moment matrices}.
For any $f(Y)\in\RR[Y]_k$, let  $\bf{f}$ denote its column vector of
coefficients  in the canonical monomial basis of
$\RR[Y]_k$. A polynomial $f(Y)\in\RR[Y]$ is said to be a
sum of squares of polynomials if it can be written as $f(Y)=\sum_{i=1}^t
f_i(Y)^2$ for some $f_1(Y),\ldots,f_t(Y)\in\RR[Y]$.  The symbol
$\Sigma^2[Y]$ denotes the set of polynomials that are s.o.s.

Let $G:=\{g_1,\ldots,g_s\}$ be the set of polynomials that define the
semialgebraic set $S$ $(\ref{eq::S})$.  We denote by
\[
  \mathcal{Q}(G):=\left\{\sum_{j=0}^s\sigma_jg_j\ \Big|\ g_0=1,\
\sigma_j \in \Sigma^2[Y], j=0,1,\ldots,s \right\}
\]
the {\itshape quadratic module} generated by $G$ and denote by
\[
\Qk(G):=\left\{\sum_{j=0}^s\sigma_jg_j\ \Big|\ g_0=1,\
\sigma_j \in \Sigma^2[Y], \,
\deg(\sigma_jg_j)\le 2k, j=0,1,\ldots,s\right\}
\]
its  $k$-th {\itshape quadratic module}.
It is clear that if $f\in\mathcal{Q}(G)$, then
$f(y)\ge 0$ for any $y\in S$. However, the converse is not necessarily
true, see Example \ref{ex::counterex}.
Note that checking whether $f\in\mathcal{Q}_k(G)$ for a fixed $k\in\N$ is an
SDP feasibility problem \cite{LasserreGlobal2001}.

For $k\in\N$, denote $s(k):={n+k\choose n}$.
Consider a finite sequence of real numbers
$z:=(z_{\alpha})_{\alpha \in \mathbb{N}^n_{2k}}\in\RR^{s(2k)}$
whose elements are indexed by $n$-tuples $\alpha\in\N^n_{2k}$. $z$ is called
a {\itshape truncated
moment sequence} up to order $2k$ if there exists
a Borel measure $\mu$ on $\RR^n$ such that
\[
	z_{\alpha}=\int Y^{\alpha}\ud\mu(y),\ \forall \alpha\in\N^n_{2k}.
\]
In this case, we say that $z$ has a {\itshape representing measure} $\mu$.
The associated $k$-th {\itshape moment matrix} is the matrix  $M_k(z)$
indexed by $\mathbb{N}^n_{k}$, with
$(\alpha,\beta)$-th entry $z_{\alpha+\beta}$ for $\alpha, \beta \in
\mathbb{N}^n_{k}$.  Given a polynomial $f(Y)=\sum_{\alpha}f_{\alpha}Y^{\alpha}$,
for $k\ge d_f:=\lceil \deg(f)/2\rceil$, the $(k-d_f)$-th {\itshape localizing
moment matrix} $M_{k-d_f}(fz)$ is defined as the moment matrix of the
{\itshape shifted vector} $((fz)_{\alpha})_{\alpha\in\mathbb{N}^n_{2(k-d_f)}}$
with $(fz)_{\alpha}=\sum_{\beta}f_{\beta}z_{\alpha+\beta}$.
${\mathscr{M}_{2k}}$ denotes the space of all sequences
$z=(z_{\alpha})_{\alpha \in \mathbb{N}^n_{2k}}\in\RR^{s(2k)}$ with
order at most $2k$.  For any $z\in\mathscr{M}_{2k}$, the corresponding
Riesz functional
$\mathscr{L}_{z}$ on $\RR[Y]_{2k}$ is  defined by
\[
\mathscr{L}_z\left(\sum_{\alpha}q_{\alpha}Y_1^{\alpha_1}\cdots
Y_n^{\alpha_n}\right)
:=\sum_{\alpha}q_{\alpha}z_{\alpha},\quad \forall q(Y)\in\RR[Y]_{2k}.
\]
From the definition of the localizing moment
matrix $M_{k-d_f}(fz)$, it is easy to check that
\begin{equation}\label{eq::M}
\mathbf{q}^TM_{k-d_f}(fz)\mathbf{q}=\mathscr{L}_z(f(Y)q(Y)^2),\quad \forall
q(Y)\in\RR[Y]_{k-d_f}.
\end{equation}
Let $d_j:=\lceil\deg(g_j)/2\rceil$ for each $j=1,\ldots,s$.
For any $v\in S$, let $\zeta_{2k,v}:=[v^{\alpha}]_{\alpha\in\N^n_{2k}}$
be the {\itshape Zeta vector} of $v$ up to degree $2k$, i.e.,
\[
	\zeta_{2k,v}=
	[1\quad v_1\quad \cdots\quad v_n\quad
		v_1^2\quad v_1v_2\quad \cdots\quad v_n^{2k}].
\]
Then, $M_k(\zeta_{2k,v})\succeq 0$ and $M_{k-d_j}(g_j\zeta_{2k,v})\succeq 0$
for $j=1,\ldots,s$. In fact, let $g_0=1$, then for each $j=0,1,\ldots,s$,
\[
	\begin{aligned}
		\mathbf{q}^TM_{k-d_j}(g_j\zeta_{2k,v})\mathbf{q}
		&=\mathscr{L}_{\zeta_{2k,v}}(g_j(Y)q(Y)^2)=g_j(v)q(v)^2\ge 0,\ \forall
		q(Y)\in\RR[Y]_{k-d_j}.
	\end{aligned}
\]
\begin{definition}\label{def::AC}
We say that ${\cal Q}(G)$
is {\itshape Archimedean }  if
there exists $\psi\in {\cal Q}(G)$ such that the inequality $\psi(y)\ge 0$
defines a compact set in $\RR^n$.
\end{definition}
Note that the  Archimedean property implies that $S$ is compact
but the converse is not
necessarily true. However, for any compact set $S$ we can always
force the associated quadratic module to be Archimedean
by adding a
redundant constraint $M-\Vert y\Vert^2_2\ge 0$ in the description of
$S$ for sufficiently large $M$.

\begin{theorem}\label{th::PP}{\upshape\cite[{Putinar's Positivstellensatz\/}]{Putinar1993}}
Suppose that $\mathcal{Q}(G)$ is Archimedean.
\begin{enumerate}[\upshape (i)]
	\item	If a polynomial $f\in\RR[Y]$ is
		positive on $S$, then $f\in\Qk(G)$ for some $k\in\N$;
	\item If $M_k(z)\succeq 0$ and $M_k(g_jz)\succeq 0$ for all
		$j=1,\ldots,s$, and all $k=0,1,\ldots$, then 
		$z=(z_{\alpha})_{\alpha \in \mathbb{N}^n}\in\RR^{\N^n}$
		has a representing measure $\mu$ supported by $S$.
\end{enumerate}
\end{theorem}

\subsection{Dual problems and GPM}\label{sec::dual}

The Lagrangian dual problem
\cite{Hettich1993,Still2007,ShapiroSIP} of $(\ref{eq::lsipp})$ is 
\begin{equation}
\label{eq::Mdual}
\left\{
\begin{aligned}
	d^*:=\sup_{\mu\in M^+(S)}&\ -\int_{S} b(y) \ud\mu(y)\\
	\text{s.t.}&\ \ \int_{S} a_i(y) \ud\mu(y)=c_i, \ i=1,\ldots,m,
\end{aligned}\right.
\end{equation}
where $M^+(S)$ is the space of all nonnegative bounded regular Borel
measure supported by $S$. The dual problem \eqref{eq::Mdual} is in
fact a special case
of the so-called generalized problems of moments (GPM), which is to
maximize a linear function over a linear section of the moment cone.
We refer the interested readers to \cite{CMP1965,AK1962,MM} and the
references therein for
various methodologies and applications of GPM problems. For numerical
treatment of GPM problems, see \cite{MPTAT,K1980} for some geometric
approaches and \cite{Lasserre2008,Nie201501} for SDP relaxation
methods for GPM problems with polynomial data.

Now we introduce the main idea of the SDP relaxation method for
\eqref{eq::Mdual} proposed by Lasserre in \cite{Lasserre2008}. 
Assume that $S$ is compact,
by Putinar's Positivstellensatz (part $(ii)$ of Theorem \ref{th::PP}),
a sequence $z=(z_{\alpha})_{\alpha \in \mathbb{N}^n}\in\RR^{\N^n}$ has
a representing measure $\mu$ supported by $S$ if
\[
M_k(z)\succeq 0,\quad M_k(g_jz)\succeq 0,\ j=1,\ldots,s,\ k=0,1,
\ldots.
\]
Define
\begin{equation}\label{eq::degree}
		\begin{aligned}
			&d_j:=\lceil\deg(g_j)/2\rceil, \ d_S:=\max\{1, d_1,\ldots,d_s\}, \\
			&d_P:=\max\{d_S,\lceil\deg(a_1)/2\rceil,\cdots,\lceil\deg(a_m)/2\rceil,
			\lceil\deg(b)/2\rceil\}.
		\end{aligned}
	\end{equation}
	Let $k\ge d_P$ and
	$z=(z_\alpha)_{\alpha\in\N^n_{2k}}\in\RR^{s(2k)}$, 
the $k$-th semidefinite relaxation of $(\ref{eq::Mdual})$ is 
\begin{equation}
	\left\{
	\begin{aligned}
		p^{\mom}_k:=\sup_{z\in\RR^{s(2k)}}
		&\ -\sum_{\alpha\in\N^n_{2k}} b_\alpha z_\alpha\\
		\text{s.t.}&\ \sum_{\alpha\in\N^n_{2k}} a_{i,\alpha} z_\alpha=c_i,\ \
		i=1,\ldots,m,\\
		&\ M_k(z)\succeq 0,\ M_{k-d_j}(g_jz)\succeq 0,\ j=1,\ldots,s.
	\end{aligned}\right.
	\label{eq::dsdplsipp}
\end{equation}
Under certain assumptions, Lasserre proved \cite{Lasserre2008} that
$p^{\mom}_k$ decreasingly converges to $p^*$. 
The SDP relaxations $(\ref{eq::dsdplsipp})$ can be easily implemented
and solved by the software GloptiPoly \cite{gloptipoly}
developed by Henrion, Lasserre and L{\"o}fberg.
\begin{condition}\label{con::extension}
	An optimizer $z^*$ of  the $k$-th SDP relaxation
	$(\ref{eq::dsdplsipp})$ satisfies the {\itshape flat extension condition} when
	\[
		\rank M_{k-d_S}(z^*)=\rank M_k(z^*).
	\]
\end{condition}
Based on \cite[Theorem 1.1]{CFKmoment}, Lasserre
	\cite[Theorem~2]{Lasserre2008} showed that the finite
	convergence of \eqref{eq::dsdplsipp}
	happens at order $k$ if the flat extension condition holds.

\section{SDP relaxations of LSIPP}\label{section::SDPrelax}
In this section, we present a hierarchy of SDP relaxations for LSIPP
problems. These SDP relaxations can be seen as the dual of Lasserre's
relaxations for GPM and enjoy several desirable features. For example,
(approximate) mininizers can be extracted and the convergence rate
can be estimated by using some existing results. We shall also see in
Section~\ref{sec::extension} that these
SDP relaxations can be easily extended to more general
semi-infinite programming problems.
\subsection{SDP relaxations of LSIPP problems}\label{subsec::sdp}
We assume that $S$ in $(\ref{eq::lsipp})$ is compact.
For a given feasible point $x\in\RR^m$ of the LSIPP problem
$(\ref{eq::lsipp})$,
the constraint requires that the polynomial
$a(Y)^Tx+b(Y)\in\RR[Y]$ is nonnegative on
$S$. Since every polynomial in the quadratic module $\Qm(G)$ of $S$
generated by $G$ is nonnegative on $S$, we can
relax the problem $(\ref{eq::lsipp})$ as follows
\begin{equation}\label{eq::qmrelax}
	p^{\sos}:=\inf_{x\in\RR^m}\ \ c^Tx\quad\text{s.t.}\
	a(Y)^Tx+b(Y)\in\Qm(G).
\end{equation}
Clearly, any feasible point of $(\ref{eq::qmrelax})$ is also feasible
for $(\ref{eq::lsipp})$. Hence, we have $p^{\sos}\ge p^*$.

\begin{define}\label{def::slater}
	We say that the Slater condition holds for the problem $(\ref{eq::lsipp})$	
	if there exists $\bar{x}\in\RR^m$ such that $a(y)^T\bar{x}+b(y)>0$
	for all $y\in S$.
\end{define}

\begin{theorem}\label{th::equi}
	If $\Qm(G)$ is Archimedean and
	the Slater condition holds for the LSIPP problem
	$(\ref{eq::lsipp})$, then $p^{\sos}=p^*$.
\end{theorem}
\begin{proof}
	Fix an $\varepsilon>0$ and a feasible $\bar{x}\in\RR^m$ of (\ref{eq::lsipp})
	such that $a(y)^T\bar{x}+b(y)>0$ for all $y\in S$.
	We next show that $p^{\sos}-p^*<\varepsilon$.
	By Putinar's Positivstellensatz,
	$\bar{x}$ is a feasible point of (\ref{eq::qmrelax})
	and thus we can assume that $c\neq 0$ without loss of generality.
	If $c^T\bar{x}-p^*<\varepsilon$,
	then $p^{\sos}-p^*\le c^T\bar{x}-p^*<\varepsilon$ and we are done.
	Hence, we assume that $c^T\bar{x}-p^*\ge\varepsilon$ in the following.
	Then we can fix another feasible point $x'\in\RR^m$ of (\ref{eq::lsipp})
	such that $c^T\bar{x}>c^Tx'$ and $c^Tx'-p^*<\varepsilon/2$.
	Let
	\begin{equation}\label{eq::delta}
		\delta:=\frac{\varepsilon}{2c^T(\bar{x}-x')}>0\quad
		\text{and}\quad \hat{x}:=(1-\delta)x'+\delta\bar{x}.
	\end{equation}
	Then we have $0<\delta<1$ and hence
	\begin{equation}\label{eq::hat}
		a(y)^T\hat{x}+b(y)=(1-\delta)[a(y)^Tx'+b(y)]+\delta
		[a(y)^T\bar{x}+b(y)]>0,\ \ \forall y\in S.
	\end{equation}
	Since $\Qm(G)$ is Archimedean, $a(Y)^T\hat{x}+b(Y)\in\Qm(G)$
	by Putinar's Positivstellensatz.
	That is, $\hat{x}$ is feasible for both $(\ref{eq::lsipp})$ and
	$(\ref{eq::qmrelax})$.
	We have
	\begin{equation}\label{eq::inequality}
		\begin{aligned}
			p^{\sos}-p^*&\le c^T\hat{x}-p^*\\
			&=(1-\delta)c^Tx'+\delta c^T\bar{x}-p^*\\
			&=(c^Tx'-p^*)+\delta c^T(\bar{x}-x')\\
			&<\frac{\varepsilon}{2}+\frac{\varepsilon}{2}
			=\varepsilon,
		\end{aligned}
	\end{equation}
	which means that $p^{\sos}\le p^*$ since $\varepsilon>0$ is arbitrary.
	As $p^{\sos}\ge p^*$, we can conclude that $p^{\sos}=p^*$.
%
%
\end{proof}

Note that we do not require that $p^*$ is attainable in the above proof.
For $k\ge d_P$,
replacing $\Qm(G)$ in $(\ref{eq::qmrelax})$ by its $k$-th truncation
$\Qm_k(G)$, we obtain
\begin{equation}\label{eq::psdplsipp}
	\left\{
	\begin{aligned}
		p^{\sos}_{k}:=\inf_{x\in\RR^m}&\ \ c^Tx\\
		\text{s.t.}&\ a(Y)^Tx+b(Y)=
		\sum_{j=0}^s\sigma_j(Y)g_j(Y),\\
		&\ \ g_0=1, \sigma_j\in\Sigma^2[Y], \deg(\sigma_jg_j)\le 2k,\
		j=0,\ldots,s.
	\end{aligned}\right.
\end{equation}
Now we reformulate $(\ref{eq::psdplsipp})$
as an SDP problem.
For any $t\in\N$,
let $m_{t}(Y)$ be the column vector consisting of all the monomials
in $Y$ of degree up to $t$. Recall that $s(t)={n+t\choose n}$ which
is the
dimension of $m_t(Y)$.  For each $j=0,1,\ldots,s$, there
exists a positive semidefinite matrix
$Z_j\in\RR^{s(k-d_j)\times s(k-d_j)}$
such that
\[
	\sigma_j(Y)=m_{k-d_j}(Y)^T\cdot Z_j\cdot m_{k-d_j}(Y).
\]
For each $\alpha\in\N^n_{2k}$, we can find a symmetric matrix
$C_{j,\alpha}\in\RR^{s(k-d_j)\times s(k-d_j)}$ such that the coefficient
of $\sigma_jg_j$ equals $\langle Z_j, C_{j,\alpha}\rangle$ for each
$j=0,1,\ldots,s$.
Let
\[
	b(Y)=\sum_{\alpha\in\N^n_{2k}} b_{\alpha}Y^\alpha\quad\text{and}\quad
	a_i(Y)=\sum_{\alpha\in\N^n_{2k}} a_{i,\alpha}Y^\alpha,\  i=1,\ldots,m.
\]
Then $(\ref{eq::psdplsipp})$ can be written as the following SDP problem
\begin{equation}\label{eq::SDPform}
	\left\{
	\begin{aligned}
		p^{\sos}_k=\inf_{Z_j\succeq 0, x\in\RR^m}&\ \ c^Tx\\
		\text{s.t.}&\ \ \sum_{i=1}^m x_ia_{i,\alpha}+
		b_{\alpha}
		=\sum_{j=0}^s\langle Z_j, C_{j,\alpha}\rangle,
		\ \forall \alpha\in\N^n_{2k}.
	\end{aligned}\right.
\end{equation}
It follows that
\begin{theorem}\label{th::convergence}
	If $\Qm(G)$ is Archimedean and
	the Slater condition holds for the LSIPP problem
	$(\ref{eq::lsipp})$, then
	$p^{\sos}_k$ decreasingly converges to $p^*$
	as $k\rightarrow\infty$.
\end{theorem}
\begin{proof}
	For any $\varepsilon>0$, let $\hat{x}$ be defined as 
	in the proof of Theorem \ref{th::equi}.
	We have $a(Y)^T\hat{x}+b(Y)\in\Qm_k(G)$ for some $k\in\N$ and then
	$p^{\sos}_k-p^*\le c^T\hat{x}-p^*<\varepsilon$. Since $\varepsilon$ is
	arbitrary, $p^{\sos}_k$ decreasingly converges to $p^*$
	as $k\rightarrow\infty$.
\end{proof}

The Lagrangian dual problem of $(\ref{eq::psdplsipp})$ is eactly the
SDP relaxation \eqref{eq::dsdplsipp} derived by Lasserre in
\cite{Lasserre2008}.
By the `weak duality', we have $p^{\mom}_k\le p^{\sos}_k$.
Consequently, we can reprove the convergence of \eqref{eq::dsdplsipp}. 
\begin{theorem}\label{th::moncon}
	If $\Qm(G)$ is Archimedean and
	the Slater condition holds for the LSIPP problem
	$(\ref{eq::lsipp})$, then
	$p^{\mom}_k$ decreasingly converges to $p^*$
	as $k\rightarrow\infty$.
\end{theorem}
\begin{proof}
Since $S$ is compact and the Slater condition holds for $(\ref{eq::lsipp})$,
$p^*=d^*$ and $d^*$ is attainable (c.f. \cite{CCK1965}). 
It is clear that $p^{\mom}_k\ge d^*=p^*$ for each $k\ge d_P$. Then the
conclusion follows from Theorem~\ref{th::convergence} and the `weak
duality'. 
\end{proof}

For any feasible point $x\in\RR^m$ of $(\ref{eq::lsipp})$, the
{\itshape active index set} of $x$ is
\[
	\{y\in S\mid a(y)^Tx+b(y)=0\}.
\]
Consider the flat extension condition (Condition~\ref{con::extension}). If it
happens, then $p^{\mom}_k=p^*$ and by \cite[Theorem
1.1]{CFKmoment}, $z^*$ has a unique $r$-atomic measure
	supported by $S$, i.e., there exist $r$ positive real numbers
	$\lambda_1,\ldots,\lambda_r$ and
	$r$ distinct points $v_1,\ldots,v_r\in S$ such that
	\begin{equation}\label{eq::v}
		z^*=\lambda_1\zeta_{2k,v_1}+\cdots+\lambda_r\zeta_{2k,v_r},
	\end{equation}
	where $\zeta_{2k,v_i}$ is the Zeta vector of $v_i$ up to degree $2k$.
\begin{prop}\label{pro::certificate}
	Suppose that $\Qm(G)$ is Archimedean and
	the Slater condition holds for the LSIPP problem
	$(\ref{eq::lsipp})$.
	Then, $v_1,\ldots,v_r$ in $(\ref{eq::v})$
	belong to the active index set
	of each minimizer $x^*$ of
	$(\ref{eq::lsipp})$.
\end{prop}
\begin{proof}
	As
		\[
		p^*=c^Tx^*=\sum_{i=1}^r \lambda_i a(v_i)^Tx^*\ge -\sum_{i=1}^r \lambda_i b(v_i)
		=p^{\mom}_k=p^*
	\]
	for any minimizer $x^*$ of $(\ref{eq::lsipp})$, the conclusion follows.
\end{proof}

The extraction procedure of the points $v_i$'s
can be found in \cite{LasserreHenrion} and has been implemented in
GloptiPoly.
\begin{remark}\label{rk::flattruncation}
	Note that the flat extension condition is only a
	sufficient condition which means that it might not hold when the
	finite convergence of \eqref{eq::dsdplsipp} happens. 
	A weaker stopping criterion called
	{\itshape flat truncation condition} was proposed by Nie in
	\cite{NieFlatTruncation} for SDP relaxations of polynomial optimziaiton problems. 
	It can also be used as a sufficient
	condition to certify the finite convergence of
	\eqref{eq::dsdplsipp}. Precisely, if an optimizer $z^*$ of  the
	$k$-th SDP relaxation $(\ref{eq::dsdplsipp})$ satisfies 
	\[
		\rank M_{t-d_S}(z^*)=\rank M_t(z^*)
	\]
	for some integer $t\in[d_P, k]$, then $p^{\mom}_k=p^*$ and the
	points $v_1,\ldots,v_r$ can also be extracted. See
	\cite{NieFlatTruncation} for details.
\end{remark}

\vskip 5pt
Compared with existing numerical approaches for LSIP problems, the SDP
relaxations (\ref{eq::psdplsipp}) and
(\ref{eq::dsdplsipp}) are applicable for LSIPP problems with index sets
being arbitrary basic semialgebraic sets, not necessarily box-shaped. 
\begin{example}\label{ex::ex2}
Consider the following problem
\begin{equation}\label{eq::ex2}
	\left\{
	\begin{aligned}
		\inf_{x\in\RR^2}&\ \ x_2\\
		\text{s.t.}&\ \ x_1y_1+x_2-y_2\ge 0, \
		\forall y\in S,
	\end{aligned}\right.
\end{equation}
where
\[
	S:=\{y\in\RR^2\mid (y_1+5y_2)y_1^2-(y_1^2+y_2^2)^2\ge 0\}
\]
which is the gray region in Figure \ref{fig::ex2}.
Clearly, it is equivalent to the bilevel problem
\[
	\min_{x_1\in\RR}\max_{y\in S}\ y_2-x_1y_1.
\]
By replacing the lower level maximality condition by the KKT condition, it
is easy to check that
the minimizer is $x^*=(\frac{1}{5},\frac{125}{104})$ and its active index set
consists of
\begin{equation*}\label{eq::tangentpoints}
	\begin{aligned}
		&\left(\frac{625}{2704}+\frac{1875}{2704}\sqrt{3},\
		\frac{3375}{2704}+\frac{375}{2704}\sqrt{3}\right)
		\approx (1.4322, 1.4884),\\
		&\left(\frac{625}{2704}-\frac{1875}{2704}\sqrt{3},\
		\frac{3375}{2704}-\frac{375}{2704}\sqrt{3}\right)
		\approx (-0.9699, 1.0079).
	\end{aligned}
\end{equation*}
Thus, the optimum is $\frac{125}{104}\approx 1.2019$.
Using GloptiPoly, we get $p^{\mom}_2=1.2982$ and $p^{\mom}_3=1.2019$.
The flat extension condition holds at the order $k=3$. 
We can extract the active index set $\{(1.4321,1.4883), (-0.9699,1.0079)\}$.
\hfill$\square$
\begin{figure}
\centering
\includegraphics[width=0.35\textwidth]{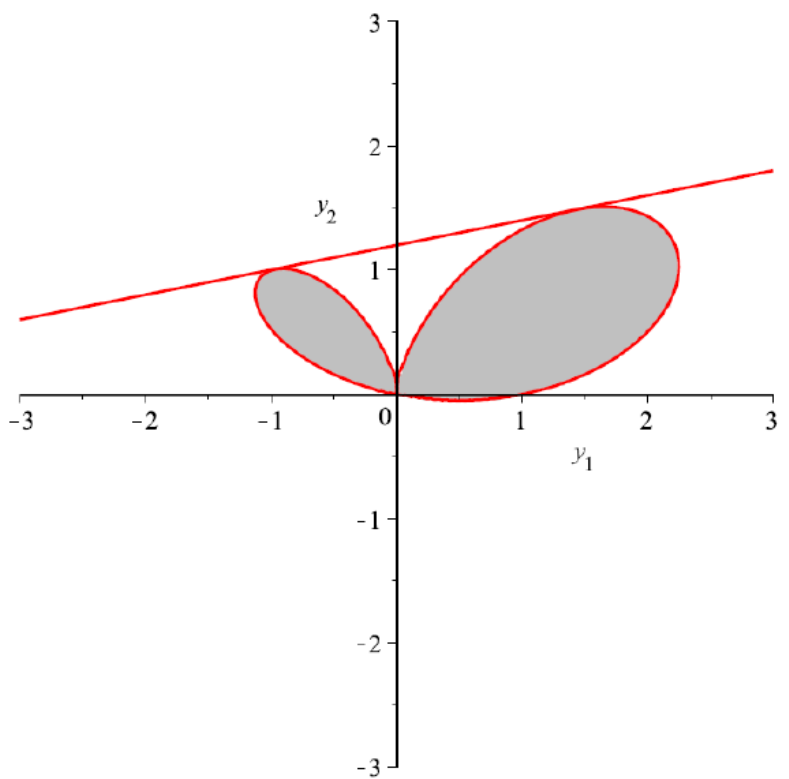}
\caption{The semialgebraic set $S$ (gray) in Example \ref{ex::ex2} and the
line $x_1^*y_1+x_2^*-y_2=0$ (red).}\label{fig::ex2}
\end{figure}
\end{example}

Although the optimal value $p^*$ of \eqref{eq::lsipp} can be
approximated by solving the dual problem \eqref{eq::Mdual} with the
SDP relaxations \eqref{eq::dsdplsipp} given in \cite{Lasserre2008}, the
hierarchy of SDP relaxations
\eqref{eq::psdplsipp} of \eqref{eq::lsipp} itself is of independent
interest.
%
For example, we can solve the relaxation $(\ref{eq::psdplsipp})$ and
extract the optimal solution $(x^{(k)},Z_0^{(k)},\ldots,Z_s^{(k)})$ (if it exists)
by the software YALMIP \cite{YALMIP}. As
$p^{\sos}_k$ may not be attainable,
let $(\tilde{x}^{(k)},\tilde{Z}_0^{(k)},\ldots,\tilde{Z}_s^{(k)})$
be an $\frac{\varepsilon}{2}$-optimal solution of (\ref{eq::psdplsipp}).
Since $\tilde{x}^{(k)}$ is feasible for (\ref{eq::lsipp}) and
$p^{\sos}_k-p^*<\frac{\varepsilon}{2}$ for some $k\in\N$, a subsequence
of $\{\tilde{x}^{(k)}\}_{k\in\N}$ converges to an $\varepsilon$-optimal solution
of (\ref{eq::lsipp}) if the feasible set of (\ref{eq::lsipp}) is bounded.

\begin{example}\label{ex::af}
Consider the following problem	
\begin{equation}\label{eq::CAP2}
	\left\{
	\begin{aligned}
		\min_{x_0, x_{i_1,i_2}\in\RR}&\ x_0\\
		\text{s.t.}&\ \left|\sum_{i_1=0}^t\sum_{i_2=0}^tx_{i_1,i_2}y_1^{i_1}y_2^{i_2}-b(y)\right|\le x_0,
		\quad \forall y\in [-1,1]^2,\\
	\end{aligned}\right.
\end{equation}
which is to approximate the function $b(Y)$ from the spans of
$Y_1^{i_1}Y_2^{i_2}$ in some sense.
Hence, it is more useful to give the minimizers $x^*$ which are the corresponding
optimal coefficients for the basis functions $Y_1^{i_1}Y_2^{i_2}$ in the approximations.
Here, we consider two cases \cite{Watson1975}: 
\[
	\text{(i):}\	t=2,\ b(Y)=\frac{1}{Y_1+2Y_2+4};\quad 
	\text{and\  (ii):}\ t=2,\ b(Y)=\sqrt{Y_1+2Y_2+4}.
\]
While $b(Y)$ in (ii) is a semialgebraic function,
\eqref{eq::psdplsipp} still works by adding some lifted variables, see
Section~\ref{sec::extension1}. 
Solving \eqref{eq::psdplsipp} with YALMIP, the obtained coefficients
are
listed below
\[
	\begin{aligned}
		\text{(i)}:\ &(0.2341,-0.0468,-0.1507,0.1203,-0.0706,-0.1292,0.0837,0.0136,0.0927),\\
		\text{(ii)}:\
		&(2.0043,0.2494,0.5125,-0.0747,0.0194,0.0350,-0.0133,-0.0178,-0.0708),
	\end{aligned}
\]
in the order
$(x^*_{0,0},x^*_{1,0},x^*_{0,1},x^*_{1,1},x^*_{2,1},x^*_{1,2},x^*_{2,2},x^*_{2,0},x^*_{0,2})$.
\hfill$\square$
%
%
%
%
%
%
%
\end{example}


At the end of this part, 
let us briefly introduce the SDP
relaxation method for general SIPP problems given in \cite{SIPPSDP}
which can also be used to solve \eqref{eq::lsipp}. The approach in
\cite{SIPPSDP} is based on the following exchange scheme. At its
$k$-th iteration, we solve the linear programming problem
\begin{equation}\label{eq::lp}
	\min_{x\in\RR^m}\ \ c^Tx\quad\text{s.t.}\ a(y)^Tx+b(y)\ge 0, \
	y\in Y_k, 
\end{equation}
where $Y_k\subset S$ is a finite set and generated at
the last iteration.
Next, choose a minimizer $x^{(k)}$ of \eqref{eq::lp} and globally
solve the polynomial optimizaion problem
\begin{equation}\label{eq::posipp}
	\min_{y\in S}\ \ a(y)^Tx^{(k)}+b(y)
\end{equation}
by Lasserre's SDP relaxation method \cite{LasserreGlobal2001}. Extract
the set $\mathcal{S}_k$ of global minimizers of \eqref{eq::posipp} and
let $Y_{k+1}=Y_k\cup\mathcal{S}_k$, then go to next iteration. To
guarantee the convergence, the feasible set of \eqref{eq::lsipp} need
to be
compact and $\mathcal{S}_k\neq\emptyset$ for each $k$. However, if the flat
extension condition or flat trucation condition is not satisfied when
solving \eqref{eq::posipp} by
Lasserre's relaxation (when $\mathcal{S}_k$ is infinite, for example),
it is hard to obtain the set $\mathcal{S}_k$ without which the iteration goes
into dead loop.

\subsection{Convergence rate analysis}\label{sec::cra}
Denote by $\mathcal{F}$ and $\mathcal{F}_k$ the feasible sets of
\eqref{eq::lsipp} and \eqref{eq::psdplsipp}, respectively. 
In this subsection, we assume that $\mathcal{F}$ and $S$ are compact.
For the simplicity in the convergence rate analysis of the SDP
relaxations \eqref{eq::psdplsipp}, we consider the following
assumption which holds possibly after some rescaling.
\begin{ass}\label{ass::fs}
It holds that $\mathcal{F}\subseteq(-1,1)^m$ and $S\subseteq(-1,1)^n$
for \eqref{eq::lsipp}. 
\end{ass}

Let
\[
\omega:=\max\{\deg(a_1),\ldots,\deg(a_m),\deg(b)\}.
\]
For a polynomial $h(Y)=\sum_\alpha h_\alpha Y^\alpha\in\RR[Y]$, define
the norm
\begin{equation*}\label{eq::fnorm}
	\Vert h\Vert:=\max_\alpha\frac{\vert
	h_\alpha\vert}{\tbinom{\Vert\alpha\Vert_1}{\alpha}}. 
\end{equation*}
Recalling the proof of Theorem~\ref{th::equi}, we have 
\begin{theorem}\label{th::rate0}
	Suppose that Assumption~\ref{ass::fs} holds, $\mathcal{Q}(G)$ is
	Archimedean and
	$\bar{x}$ is a Slater point of \eqref{eq::lsipp}. Let $r_{\bar{x}}^*:=\min_{y\in
	S}a(y)^T\bar{x}+b(y)>0$,  then there
	exists $\gamma>0$ depending on $g_i$'s in \eqref{eq::S} such that
	for any $\varepsilon>0$, it holds that $0\le
	p_k^{\sos}-p^*\le\varepsilon$ whenever
	\[
		k\ge \gamma\exp\left[\left(\omega^2n^\omega\frac{\sum_{i=1}^m\Vert
		a_i\Vert+\Vert
	b\Vert}{\kappa(\varepsilon)r_{\bar{x}^*}}\right)^\gamma\right],
	\]
	where
	$\kappa(\varepsilon):=\min\{1,\frac{\varepsilon}{4\sqrt{m}\Vert
	c\Vert_2}\}$. 
\end{theorem} 
\begin{proof}
	If $c^T\bar{x}-p^*\le\varepsilon$, let
	$x^{(\varepsilon)}=\bar{x}$; otherwise, let
	$x^{(\varepsilon)}=\hat{x}$ as defined in \eqref{eq::delta}. As
	$\mathcal{F}\subseteq(-1,1)^m$, we have
	$\delta\ge\frac{\varepsilon}{4\sqrt{m}\Vert c\Vert_2}$ in
	\eqref{eq::delta}. In either case, it holds from \eqref{eq::hat} that
	$a(y)^Tx^{(\varepsilon)}+b(y)\ge\kappa(\varepsilon)r_{\bar{x}^*}>0$
	for any $y\in S$.
	By \cite[Theorem~6]{NieSchweighofer},  there
	exists $\gamma>0$ depending on $g_i$'s such that
	$a(Y)^Tx^{(\varepsilon)}+b(Y)\in\mathcal{Q}_{\bar{k}}(G)$ where 
	\[
		\bar{k}:=\gamma\exp\left[\left(\omega^2n^\omega\frac{\Vert
		a(Y)^Tx^{(\varepsilon)}+b(Y)\Vert}{\min_{y\in
		S}a(y)^Tx^{(\varepsilon)}+b(y)}\right)^\gamma\right]. 
	\]
Clearly, 
\[
	\bar{k}\le \gamma\exp\left[\left(\omega^2n^\omega\frac{\sum_{i=1}^m\Vert
		a_i\Vert+\Vert
	b\Vert}{\kappa(\varepsilon)r_{\bar{x}^*}}\right)^\gamma\right]\le
	k.
\]
Hence, $a(Y)^Tx^{(\varepsilon)}+b(Y)\in\mathcal{Q}_k(G)$ and
$p_k^{\sos}-p^*\le c^Tx^{(\varepsilon)}-p^*\le\varepsilon$ by
\eqref{eq::inequality}.
\end{proof}

For any $\varepsilon>0$, compared with \eqref{eq::lsipp}, consider the
problem
\begin{equation}\label{eq::lsippe}
(P_\varepsilon)\qquad	
\left\{
\begin{aligned}
	p^*_\varepsilon:=\inf_{x\in\RR^m}&\ c^Tx\\
	\text{s.t.}&\ a(y)^Tx+b(y)\ge \varepsilon,\ \ \forall y\in S. 
\end{aligned}\right.
\end{equation}
Obviously, $p^*\le p^*_\varepsilon$ for any $\varepsilon>0$. 
Moreover, by the stability of optimal values of linear
semi-infinite programming problems (c.f. \cite[Theorem
5.1.5]{GL2014}), it follows that
\begin{lemma}\label{lem::perturbation}
	If Assumption~\ref{ass::fs} and
	the Slater condition hold for \eqref{eq::lsipp}, then there
	exist scalars $\bar{\varepsilon}>0$ and $L>0$ such that
	$p^*_\varepsilon-p^*\le L\varepsilon$ for any
	$\varepsilon\le\bar{\varepsilon}$. 
\end{lemma}
\begin{proof}
	Since $\mathcal{F}$ is compact, the optimal solution set of
	\eqref{eq::lsipp} is non-empty and compact. 
	As $S$ is compact, a Slater point of \eqref{eq::lsipp} is also a
	strong Slater point. Then, the conclusion follows by the Lipschitz
	continuity of the optimal value function of \eqref{eq::lsipp}
	(see, \cite[Theorem 5.1.5]{GL2014}). 
\end{proof}

For any $\varepsilon>0$, denote the feasible set of
\eqref{eq::lsippe} by
\[
	\mathcal{F}_\varepsilon:=\{x\in\RR^m\mid
	a(y)^Tx+b(y)\ge \varepsilon,\quad \forall\ y\in S\}. 
\]
\begin{lemma}\label{lem::conclude}
	Suppose that Assumption~\ref{ass::fs} holds and $\mathcal{Q}(G)$
	is Archimedean. Then, there
	exists some $\gamma>0$ depending on $g_i$'s in \eqref{eq::S} such that
	for all integers $k>\gamma\exp((2\omega^2n^\omega)^\gamma)$,
	we have $\mathcal{F}_\varepsilon\subseteq\mathcal{F}_k$ whenever
	\begin{equation}\label{eq::ek}
		\varepsilon\ge \varepsilon_k:=\frac{6\omega^3n^{2\omega}(\sum_{i=1}^m\Vert
	a_i\Vert+\Vert b\Vert)}{\sqrt[\gamma]{\log\frac{k}{\gamma}}}. 
\end{equation}
\end{lemma}
\begin{proof}
	Fix a point $u\in\mathcal{F}_\varepsilon$. Let $r_u^*:=\min_{y\in
	S} a(y)^Tu+b(y)$. Then, by \cite[Theorem~8]{NieSchweighofer},
	there exists some $\gamma>0$ depending on $g_i$'s such that for all
	$k>\gamma\exp((2\omega^2n^\omega)^\gamma)$, it holds that
	\[
		a(Y)^Tu+b(Y)-r_u^*+\frac{6\omega^3n^{2\omega}\Vert
	a(Y)^Tu+b(Y)\Vert}{\sqrt[\gamma]{\log\frac{k}{\gamma}}}\in\mathcal{Q}_k(G). 
	\]
	As $u\in\mathcal{F}_\varepsilon$, we have $r_u^*\ge\varepsilon$.
	Since $\mathcal{F}_\varepsilon\subseteq\mathcal{F}$, the
	assumption $\mathcal{F}\subseteq(-1,1)^m$ implies that $\Vert
	a(Y)^Tu+b(Y)\Vert\le \sum_{i=1}^m\Vert a_i(Y)\Vert+\Vert
	b(Y)\Vert$. Consequently,
	\[
		r_u^*-\frac{6\omega^3n^{2\omega}\Vert
	a(Y)^Tu+b(Y)\Vert}{\sqrt[\gamma]{\log\frac{k}{\gamma}}}\ge
	\varepsilon-\varepsilon_k\ge 0. 
	\]
	Hence, we have $a(Y)^Tu+b(Y)\in\mathcal{Q}_k(G)$ and
	$u\in\mathcal{F}_k$. 
\end{proof}

\begin{theorem}\label{th::rate}
	Suppose that Assumption~\ref{ass::fs}, the Slater condition
	hold for \eqref{eq::lsipp} and $\mathcal{Q}(G)$ is Archimedean.
	Then, there
	exist some $\gamma>0$ depending on $g_i$'s in \eqref{eq::S} and scalars 
	$\bar{\varepsilon}>0$, $L>0$ such that for all integers
	\[
		k>\max\left\{\gamma\exp((2\omega^2n^\omega)^\gamma),\ 
			\gamma\exp\left[\frac{6\omega^3n^{2\omega}(\sum_{i=1}^m\Vert
			a_i\Vert+\Vert b\Vert)}{\bar{\varepsilon}}\right]^\gamma\right\},
    \]
    it holds that
	\[
		0\le p^{\sos}_k-p^*\le L\frac{6\omega^3n^{2\omega}(\sum_{i=1}^m\Vert
	a_i\Vert+\Vert b\Vert)}{\sqrt[\gamma]{\log\frac{k}{\gamma}}}.
	\]
\end{theorem}
\begin{proof}
	Note that all assumptions in Lemma~\ref{lem::perturbation} and
	\ref{lem::conclude} hold. Then, there exist $\gamma>0$ depending on
	$g_i$'s, $\bar{\varepsilon}>0$ and $L>0$ as described
	in the conclusions of Lemma~\ref{lem::perturbation} and
	\ref{lem::conclude}. Recall $\varepsilon_k$ defined in
	\eqref{eq::ek}. 
	By Lemma~\ref{lem::conclude}, it holds
	that $\mathcal{F}_{\varepsilon_k}\subseteq\mathcal{F}_k$ which
	implies $p^*_{\varepsilon_k}\ge p^{\sos}_k$. Moreover, 
	it is easy to check that $\varepsilon_k\le\bar{\varepsilon}$ and hence
	$p_{\varepsilon_k}^*-p^*\le L\varepsilon_k$ by
	Lemma~\ref{lem::perturbation}. Consequently, $p_k^{\sos}-p^*\le
	L\varepsilon_k$. 
\end{proof}

\subsection{On compactness of $\mathcal{F}$}\label{sec::compacttest}
In the last subsection, we assume that the feasible set $\mathcal{F}$ of
\eqref{eq::lsipp} is compact in order to estimate the convergence
rate of the SDP relaxations \eqref{eq::psdplsipp}. In the
following, we show that the compactness of $\mathcal{F}$ can be
determined by solving some LSIPP problems, which can be done by the
SDP relaxations \eqref{eq::psdplsipp} in some cases. Denote by
$\mathbf{0}$ the vector of all zeros in $\RR^n$.
In this subsection, without loss of generality,we assume that 
\begin{ass}\label{ass::0}
	$\mathbf{0}\in\mathcal{F}$, or equivalently, $b(y)\ge 0$ for all $y\in S$. 
\end{ass}

Denote by $0^+\mathcal{F}\subset\RR^m$ the {\itshape recession cone} of
$\mathcal{F}$, i.e., $u\in0^+\mathcal{F}$ if and only if $x+t
u\in\mathcal{F}$ for all $x\in\mathcal{F}$ and $t\ge 0$. As
$\mathcal{F}$ is closed and convex, $\mathcal{F}$ is compact if and
only if $0^+\mathcal{F}=\{\mathbf{0}\}$ by \cite[Theorem~8.4]{convexanalysis}. 

Consider the minimax problem
\begin{equation}\label{eq::minimax}
	r_a^*:=\max_{\Vert u\Vert_2=1}\min_{y\in S}\ a(y)^Tu
\end{equation}

\begin{prop}\label{prop::compact}
	Suppose that Assumption~\ref{ass::0} holds for \eqref{eq::lsipp}.
	Then, its feasible set $\mathcal{F}$ is compact if and only if
	$r_a^*<0$. 
\end{prop}
\begin{proof}
	As $\mathbf{0}\in\mathcal{F}$, by \cite[Theorem~8.3]{convexanalysis}, a
	vector $u\in0^+\mathcal{F}$ if and only if $t
	u\in\mathcal{F}$ for all $t\ge 0$. That is,
	$(a(y)^Tu)\cdot t+b(y)\ge
	0$ for all $y\in S$ and $t\ge 0$, which is true if and only
	if $a(y)^Tu\ge 0$ for all $y\in S$ since $b(y)\ge 0$ on $S$. 
	Note that since $S$ is compact, $\min_{y\in S}\ a(y)^Tu$ is
	continuous in $u$ and then $r_a^*$ is attainable. 
	Therefore, if $\mathcal{F}$ is compact, then
	$u\not\in0^+\mathcal{F}$ for any nonzero $u\in\RR^n$ and hence $r_a^*<0$. 
	Conversely, assume that $r_a^*<0$. 
If $u\in0^+\mathcal{F}$ for some nonzero $u\in\RR^n$, then we have
$r_a^*\ge 0$, a contradiction. Then, $0^+\mathcal{F}=\{\mathbf{0}\}$
and hence $\mathcal{F}$ is compact.  
\end{proof}
 
It is clear that the minimax problem \eqref{eq::minimax} is equivalent
to the following problem 
\begin{equation}\label{eq::eqlsipp}
	\left\{
	\begin{aligned}
		r_a^*=\sup_{\Vert u\Vert_2=1,\lambda\in\RR}&\ \lambda\\
		\text{s.t.}&\ a(y)^Tu-\lambda\ge 0,\quad \forall\ y\in S.
	\end{aligned}\right.
\end{equation}
By some rescalings, we can reformulate the problem \eqref{eq::eqlsipp}
as the following LSIPP problems of the form \eqref{eq::lsipp}:
\[
(P_i^+)\qquad	
\left\{
\begin{aligned}
	r_{a,i}^+:=\inf_{u_i=1,\lambda\in\RR}&\ -\lambda\\
	\text{s.t.}&\ a(y)^Tu-\lambda\ge 0,\ \ \forall y\in S,
\end{aligned}\right. 
\]
and 
\[
(P_i^-)\qquad	
\left\{
\begin{aligned}
	r_{a,i}^-:=\inf_{u_i=-1,\lambda\in\RR}&\ -\lambda\\
	\text{s.t.}&\ a(y)^Tu-\lambda\ge 0,\ \ \forall y\in S,
\end{aligned}\right.
\]
for $i=1,\ldots,m$. Then,
\begin{cor}\label{cor::noncompact}
	Suppose that Assumption~\ref{ass::0} holds for \eqref{eq::lsipp}.
	Then, its feasible set $\mathcal{F}$ is compact if and only if 
	$\min\{r_{a,i}^+,\ r_{a,i}^-,\ i=1,\ldots,m\}>0$. 
\end{cor}
Consequently, the compactness of $\mathcal{F}$ can be verified by a
positive lower bound of $\min\{r_{a,i}^+,\ r_{a,i}^-,\
i=1,\ldots,m\}$ which can be obtained by solving $(P^+_i)$'s and
$(P^-_i)$'s using, for instance, discretization methods.

Note that the SDP relaxations \eqref{eq::psdplsipp} and
\eqref{eq::dsdplsipp} produce {\itshape upper} bounds of $p^*$ of
\eqref{eq::lsipp}. The compactness of $\mathcal{F}$ can also be
verified by these SDP relaxations of $(P^+_i)$'s and $(P^-_i)$'s if finite
convergence happens for each problem, which can be detected by the
flat extension condition (or the weaker flat
truncation condition in Remark~\ref{rk::flattruncation}). 
In particular, when $S$ is a closed and bounded interval, the SDP
relaxations \eqref{eq::psdplsipp} and
\eqref{eq::dsdplsipp} of the smallest order are exact for
\eqref{eq::lsipp} by the representation result of nonnegative
polynomials in the univariate
case. This result has been
investigated in \cite{XSQ}. Precisely, without loss of generality, we
can assume that $S=[-1,1]$. Let 
	\begin{equation*}\label{eq::interval}
		\begin{aligned}
			&[-1,1]=\{y_1\in\RR\mid g_1(y_1)\ge
			0\},\quad\text{where}\quad g_1(Y_1)=1-Y_1^2.  
		\end{aligned}
	\end{equation*}
	Recall the well-known result
	\begin{theorem}{\upshape(c.f. \cite{Laurent_sumsof,PR2000})}\label{th::FML}
		Let $h\in\RR[Y_1]$ and $h\ge 0$ on $[-1,1]$, then
		$h=\sigma+\sigma_1(1-Y_1^2)$ where $\sigma,\sigma_1\in\Sigma^2[Y_1]$
		and $\deg(\sigma)$,
		$\deg(\sigma_1(1-Y_1^2))\le2\lceil\deg(h)/2\rceil$. 
	\end{theorem}
	It follows that $p^{\sos}_{k_P}=p^{\mom}_{k_P}=p^*$ holds for
	\eqref{eq::psdplsipp} and \eqref{eq::dsdplsipp} in this case
	\cite{XSQ}. 
	Therefore, the compactness of $\mathcal{F}$ can always be verified
	by SDP relaxations  \eqref{eq::psdplsipp} and
	\eqref{eq::dsdplsipp} of $(P^+_i)$'s and $(P^-_i)$'s when $S$ is a
	closed and bounded interval. 

\vskip 5pt
\begin{examplerv}{\ref{ex::ex2}}
	Consider the feasible set $\mathcal{F}$ of \eqref{eq::ex2} in
	Example~\ref{ex::ex2}, which is clearly noncompact. Note that
	$\mathbf{0}\not\in\mathcal{F}$ and we have proved that the
	minimizer is $(\frac{1}{5}, \frac{125}{104})$. Let
	$w_1=x_1-\frac{1}{5}$, $w_2=x_2-\frac{125}{104}$ and move the set
	$\mathcal{F}$ to 
	\[
		\left\{w\in\RR^2\mid
			\left(w_1+\frac{1}{5}\right)y_1+\left(w_2+\frac{125}{104}\right)-y_2\ge
			0,\ \forall\ y\in S\right\},
	\]
	which contains $\mathbf{0}$. Consider the LSIPP problem
	\[
        (P_1^+)\qquad	
        \left\{
        \begin{aligned}
        	r_{a,1}^+:=\inf_{u_2, \lambda\in\RR}&\ -\lambda\\
        	\text{s.t.}&\ y_1+u_2-\lambda\ge 0,\ \ \forall y\in S. 
        \end{aligned}\right. 
	\]
	It is easy to check that $y_1\ge -2$ for all $y\in S$. Then,
	$(u_2=N+2, \lambda=N)$ is feasible for $(P_1^+)$ for all $N\in\N$.
	Hence, we have $r_{a,1}^+=-\infty$ and therefore $\mathcal{F}$ is
	noncompact by Corollary~\ref{cor::noncompact}. 
\hfill$\square$
\end{examplerv}

\begin{example}
	Consider the ellipse 
	\[
		\mathcal{F}:=\{(x_1,x_2)\in\RR^2\mid 2x_1^2+x_2^2+2x_1x_2+2x_1\le 0\}
	\]
	which can be represented by
	\[
		\{(x_1,x_2)\in\RR^2\mid a(y_1)^Tx+b(y_1)\ge 0,\ \forall y_1\in S\}
	\]
	where 
	\[
		a(Y_1)=(-Y_1^4-2Y_1^3+3Y_1^2+2Y_1-1, -2Y_1(Y_1^2-1))^T,\quad
		b(Y_1)=2Y_1^2,
	\]
	and $S=[-1,1]$ $($see \cite{LSIP}$)$. Clearly, $\mathcal{F}$ is
	compact and $\mathbf{0}\in\mathcal{F}$. As $S=[-1,1]$, all
	problems $(P^+_i)$'s and $(P^-_i)$'s can be solved by the SDP
	relaxations \eqref{eq::psdplsipp} and \eqref{eq::dsdplsipp} of
	order $d_P=2$. 
	Using GloptiPoly, we first solve the SDP relaxation
	\eqref{eq::dsdplsipp} of 
	\[
        (P_1^+)\qquad	
        \left\{
        \begin{aligned}
        	r_{a,1}^+:=\inf_{u_2, \lambda\in\RR}&\ -\lambda\\
        	\text{s.t.}&\
			-y_1^4-2y_1^3+3y_1^2+2y_1-1-2y_1(y_1^2-1)u_2-\lambda\ge
			0,\\
			&\ \forall y_1\in S. 
        \end{aligned}\right. 
	\]
	As the infeasibility of the SDP problem is detected by the SDP
	solver SeDuMi \cite{Sturm99} called by GloptiPoly, we have
	$r_{a,1}^+=+\infty$. We continue to solve $(P_1^-)$, $(P_2^+)$ and
	$(P_2^-)$. The results solved by GloptiPoly are
	$r_{a,1}^-=1$, $r_{a,2}^+=0.5491$ and $r_{a,2}^-=0.7698$, which
	imply the compactness of $\mathcal{F}$ by
	Corollary~\ref{cor::noncompact}. 
\hfill$\square$
\end{example}

\vskip 10pt
Note that the index set $S$ is required to be compact to guarantee the
convergence of the SDP relaxations \eqref{eq::psdplsipp} and
\eqref{eq::dsdplsipp}. 
To end this section, 
we consider two examples to illustrate how to deal with 
the case when $S$ is noncompact by the homogenization technique
and its applications in polynomial optimzation problems.

\begin{example}\label{ex::counterex}
Consider the LSIPP problem
\begin{equation}\label{eq::counterex}
	p^*:=\inf_{x\in\RR}\ -\frac{x}{2} \quad\text{s.t.}\ (1-3y_2)x+3y_1\ge 0, \
	\forall y\in S,
\end{equation}
where
\[
	S:=\{y\in\RR^2\mid y_1\ge 0, y_1^2-y_2^3\ge 0\}.
\]
Since $(0,0)\in S$, a feasible $x$ must be nonnegative.  Clearly, $x=0$ is a feasible point.
$x>0$ is feasible if and only if
\[
	0\ge\max_{y\in S}\left\{y_2-\frac{1}{3}-\frac{y_1}{x}\right\}
	=\max_{y\in S}\left\{y_1^{\frac{2}{3}}-\frac{1}{3}-\frac{y_1}{x}\right\}.
\]
The latter maximum is attained at $\frac{8x^3}{27}$ with optimal value
$\frac{4x^2}{27}-\frac{1}{3}$. Thus, the feasible set of (\ref{eq::counterex})
is $[0,\frac{3}{2}]$ and the minimizer is $x^*=\frac{3}{2}$.
%

Obviously, $\Qm(G)$ is not Archimedean.  For any $k\in\N$,
we know from \cite[Example 2.10]{SIAMGWZ}
that $(1-3Y_2)x+3Y_1\in\Qm_k(G)$ if and only if $x=0$, i.e.,
$p^{\sos}_k=0$ for each $k\ge d_P$. Now we show that
$p^{\mom}_k=p^{\sos}_k$ for each $k\ge d_P$. In fact, for the SDP relaxation
$(\ref{eq::dsdplsipp})$ of the problem $(\ref{eq::counterex})$,
let $\mu$ be a probability measure with uniform distribution in the following
subset of $S$:
\[
	S_1:=\{(y_1,y_2)\in\RR^2\mid 1\le y_1\le 2,\ 0\le y_2\le 1\}
\]
and $z^{(\mu)}$ be the truncated moment sequence with representing measure $\mu$
up to order $2k$.
It can be verified that
$z^{(\mu)}$
is a feasible point of $(\ref{eq::dsdplsipp})$ and
its corresponding truncated moment matrix and localizing moment matrices
are positive definite since $S_1$ has nonempty interior.
Then $p^{\mom}_k=p^{\sos}_k$ follows by the conic duality theorem.
Hence, both SDP relaxations $(\ref{eq::psdplsipp})$
and $(\ref{eq::dsdplsipp})$ do not converge to the optimum.

Now let us see how to solve this issue by homogenization.
We first homogenize the defining polynomials of $S$ by new variable
$y_0$ and define the following bounded set
	\[
		\begin{aligned}
			\So&:=\{\td{y}=(y_0, y_1, y_2)\in\RR^3\mid y_1\ge 0,\ y_0y_1^2-y_2^3\ge 0,\ y_0>
0,\ \Vert \td{y}\Vert_2^2=1\}.\\
		\end{aligned}
	\]
	Then, we homogenize the constraint polynomial of
	\eqref{eq::counterex} with respect to $Y$ and consider the problem 
	\[
		\inf_{x\in\RR}\ -\frac{x}{2} \quad\text{s.t.}\ (y_0-3y_2)x+3y_1\ge 0, \
	\forall \td{y}=(y_0,y_1,y_2)\in {\sf closure}(\So),
	\]
	which is equivalent to \eqref{eq::counterex} by \cite[Proposition
	4.2]{SIPPSDP}. However, the set ${\sf closure}(\So)$ is not in the
	form of basic semialgebraic sets. Hence, we define the following
	compact set
	\[
\St:=\{(y_0, y_1, y_2)\in\RR^3\mid y_1\ge 0,\ y_0y_1^2-y_2^3\ge 0,\ y_0\ge
0,\ \Vert \td{y}\Vert_2^2=1\}.
	\]
	We say $S$ is {\itshape closed at $\infty$} \cite{exactJacNie} if
	$\St={\sf closure}(\So)$, in which case
	\eqref{eq::counterex} is equivalent to 
	\begin{equation}\label{eq::counterex2}
		\inf_{x\in\RR}\ -\frac{x}{2} \quad\text{s.t.}\ (y_0-3y_2)x+3y_1\ge 0, \
	\forall \td{y}=(y_0,y_1,y_2)\in \St.
	\end{equation}
Note that $S$ is indeed closed at $\infty$. In fact,
for every $(0, v_1, v_2)\in\St\backslash\So$, let
\[
v^{(\varepsilon)}:=\left(\varepsilon,\ v_1,\
\sqrt[3]{\varepsilon v_1^2+v_2^3}\right).
\]
Then $\{v^{(\varepsilon)}/\Vert v^{(\varepsilon)}\Vert_2\}_{\varepsilon>0}
\subseteq\So$ and
$\lim_{\varepsilon\rightarrow 0}v^{(\varepsilon)}
/\Vert v^{(\varepsilon)}\Vert_2=(0, v_1, v_2)$. Hence, we have
$\St\backslash\So\subseteq{\sf closure}(\So)$ and
so $S$ is closed at $\infty$. Clearly, the quadratic module associated
with $\St$ is Archimedean and $\bar{x}=1$ is
a 
Slater point of $(\ref{eq::counterex2})$.
With GloptiPoly, we solve the SDP relaxations \eqref{eq::dsdplsipp} of
\eqref{eq::counterex2} and get the following numerical results:
$p_2^{\mom}=-1.2124\times 10^{-8}$ and $p_3^{\mom}=-0.7500$.
The flat extension condition is satisfied for $k=3$ and we obtain
the certified optimum $-0.7500$. By Proposition \ref{pro::certificate},
the extracted numerical active index set of the minimizer $x^*=3/2$ is
${(0.5773,0.5774,0.5774)}$ which corresponds to
$(1,1)\in S$.
\hfill$\square$
\end{example}
\begin{remark}
	Note that not every set $S$ of the form $(\ref{eq::S})$ is closed at $\infty$
	even when it is compact \cite[Example 5.2]{NieDisNon}.
	However, it is shown in \cite[Theorem 4.10]{SIPPSDP} that the
	closedness at $\infty$ is a {\itshape generic} property.
\end{remark}


\begin{example}\label{ex::stable}
Consider the following polynomial optimization problem
\begin{equation}\label{eq::stable}
	\left\{
	\begin{aligned}
		\inf_{y\in\RR^2}&\ \ f(y):=y_1^2+y_2^2\\
		\text{s.t.}&\ \  y\in
	S:=\{y\in\RR^2\mid g_1(y)\ge 0,\ g_2(y)\ge 0,\ g_3(y)\ge 0\},
	\end{aligned}\right.
\end{equation}
where 
\[
g_1(Y)=Y_2^2-1,\  g_2(Y)=Y_1^2-Y_1Y_2-1,\ g_3(Y)=Y_1^2+Y_1Y_2-1. 
\]
It was shown in \cite{DNPKKT,LSTZ,exactJacNie}
that the global minimizers and global minimum are
\[
	\left(\pm\frac{1+\sqrt{5}}{2},\pm 1\right)\approx(\pm 1.618,\pm 1)
	\quad\text{and}\quad
	2+\frac{(1+\sqrt{5})}{2}\approx 3.618.
\]
Because $S$ is noncompact,
the classic Lasserre's SDP relaxations \cite{LasserreGlobal2001} of
$(\ref{eq::stable})$ can only
provide lower bounds $2$ no matter how large the order is (c.f.
\cite{DNPKKT}). 

Clearly, any polynomial optimization problem of the form
\eqref{eq::stable} can be equivalently reformulated to the following
LSIPP problem 
\begin{equation}\label{eq::po2lsipp1}
f^*=\sup_{x\in\RR}\ \ x \quad\text{s.t.}\  \ f(y)-x\ge 0, \ \forall y\in S.
\end{equation}
As $S$ is noncompact, we use the homogenization technique in
Example~\ref{ex::counterex} to convert this LSIPP problem to 
\begin{equation}\label{eq::po2lsipp}
		\td{f}^*:=\sup_{x\in\RR}\ \ x \quad 
		\text{s.t.}\ \ f^h(\td{y})-xy_0^{\deg(f)}\ge 0, \ \forall
		\td{y}\in \St, 
	\end{equation}
where $f^h$ is the homogenization of $f$ and $\St$ is
defined as in Example~\ref{ex::counterex}. 
Suppose that $f^*>-\infty$, then
the Slater condition holds for \eqref{eq::po2lsipp} if and only if 
\begin{equation}\label{con::slater2}
	\hat{f}(y)>0,\quad\forall \ y\in\widehat{S}:=\{y\in\RR^n\mid
		\hat{g}_1(y)\ge 0,\ \ldots,\ \hat{g}_s(y)\ge 0,\ \Vert
	y\Vert_2^2=1\},
\end{equation}
where $\hat{f}$ and $\hat{g}_i$'s are the homogeneous parts of $f$ and
$g_i$'s of the highest degree. 
Moreover, if the condition \eqref{con::slater2} holds for
\eqref{eq::po2lsipp}, it is easy to see that any feasible point of
\eqref{eq::po2lsipp1} is also feasible for \eqref{eq::po2lsipp}. Thus,
$\td{f}^*=f^*$ and we can
compute them by the SDP relaxations \eqref{eq::psdplsipp} and
\eqref{eq::dsdplsipp}. 

Obviously, the condition \eqref{con::slater2} holds for
\eqref{eq::stable}. We compute the relaxations \eqref{eq::dsdplsipp}
of \eqref{eq::po2lsipp} with GloptiPoly.  For $k=3$, the flat
extension condition is satisfied and we get the numerically certified optimum
$f^{\mom}_3=3.6180$. The extracted active index set
is $\{(0.4653,\pm 0.7529,\pm 0.4653)\}$ which corresponds to the set
of global minimizers $(\pm 1.6181,\pm 1)$.
\hfill$\square$
\end{example}

\begin{remark}
(i) By \cite[Theorem 5.1 and 5.3]{Marshallopti}, the condition
\eqref{con::slater2} holds if and only if $f$ is {\itshape stably bounded from
below} on $S$, i.e., $f$ remains bounded from below on $S$ for all
sufficiently small perturbations of the coefficients of $f, g_1,\ldots,g_s$.
Therefore, we give an SDP relaxation method in
Example~\ref{ex::stable} for solving the class of polynomial
optimization problems whose objective polynomials are stably bounded
from below on noncompact feasible sets; (ii) Note that the stably
boundedness from below of $f$ on $S$ is irrelevant to the closedness at $\infty$
of $S$. For example, the set $\{y\in\RR^2\mid y_2\ge y_1^2\}$ is not
closed at $\infty$ but $Y_2^2$ is stably bounded from below on it; the
set $S$ in Example~\ref{ex::counterex} is closed at $\infty$ but
$Y_1$ is not stably bounded from below on it. 
\end{remark}

\section{Some extensions}\label{sec::extension}
In this section, we discuss some extensions of the SDP relaxations
\eqref{eq::psdplsipp}  for
\eqref{eq::lsipp} in Section~\ref{section::SDPrelax} to more general
semi-infinite programming problems.
\subsection{LSIP with semi-algebraic functions}\label{sec::extension1}
Inspired by Lasserre and Putinar's work \cite{LP2010},
we would like to point out that
the SDP relaxation method proposed in this paper is applicable to
a more general subclass of LSIP problems.
Denote by $\mathcal{X}\subseteq\RR^m$ a convex polyhedron defined by finitely
many linear inequalities in the variables $X$.
Denote by $\mathcal{A}$ the algebra consisting of functions
generated by finitely many of the dyadic operations $\{+,-,/,\vee,\wedge\}$
and monadic operations $\{\vert\cdot\vert,(\cdot)^{1/p},p\in\N\}$
on polynomials in $\RR[Y]$,
where $f\vee g:=\max[f,g]$ and $f\wedge g:=\min[f,g]$ for $f,g\in\RR[Y]$.
For example,
\[
	\sqrt{\vert f(Y)\vert+g(Y)^2}\wedge\left(\frac{1}{g(Y)}\vee f(Y)\right)\in\mathcal{A}.
\]
Note that every function in $\mathcal{A}$ has a {\itshape lifted basic semi-algebraic
representation} \cite[Definition 1]{LP2010}.
Then, the SDP relaxations (\ref{eq::psdplsipp}) and (\ref{eq::dsdplsipp}) can be
extended for more general LSIP problems of the form
\begin{equation}\label{eq::lsipp2}
\left\{
\begin{aligned}
	p^*:=\inf_{x\in\mathcal{X}}&\ c^Tx\\
	\text{s.t.}&\ a^l(y)^Tx+b_l(y)\ge 0,\ \ \forall y\in S \text{ and } l=1,\ldots,t,
\end{aligned}\right.
\end{equation}
where $c\in\RR^m$, $a^l(Y)\in\mathcal{A}^m$, $b_l(Y)\in\mathcal{A}$,
$l=1,\ldots,t$ and
\begin{equation}\label{eq::Snew}
	S:=\{y\in\RR^n\mid g_1(y)\ge 0, \ldots, g_s(y)\ge 0\},
\end{equation}
where $g_j(Y)\in\mathcal{A}$, $j=1,\dots,s$. In fact, as shown in \cite{LP2010},
the nonnegativity test
of semi-algebraic functions in $\mathcal{A}$ on the set $(\ref{eq::Snew})$
can be reduced to an equivalent polynomial funcation case in a {\itshape lifted}
space by adding some new variables. For instance, with $f,h,g_1,g_2\in\RR[Y_1]$,
\[
	\sqrt{f(y_1)}-1/h(y_1)\ge 0\quad\text{on}\quad
	\{y_1\in\RR\mid \vert g_1(y_1)\vert g_2(y_1)\ge 1\}
\]
can be written as $y_2-y_3\ge 0$ on
\[
		\{y\in\RR^{4}\mid  f(y_1)=y_2^2,\ y_2\ge 0,\ h(y_1)y_3=1,\
		 y_4g_2(y_1)\ge 1,\ g_1(y_1)^2=y_4^2,\ y_4\ge 0\}.
\]
Consequently, 
the extension to $\mathcal{A}$ of Putinar's
Positivstellensatz (\cite[Theorem 2]{LP2010}) provides us
representations of each nonnegativity
constraint in (\ref{eq::Snew}) via s.o.s and the dual theory of moments.
Notice that the constraint $x\in\mathcal{X}$ is linear in $X$.
Hence,
SDP relaxations as (\ref{eq::psdplsipp}) and its dual (\ref{eq::dsdplsipp}) can be
similarly derived for (\ref{eq::lsipp2}) by lifting the parameter space.
Moreover, the convergence results and stopping criterion,
as Theorem \ref{th::convergence}, \ref{th::moncon} and Proposition \ref{pro::certificate},
can also be analogously established.
As might be expected, additional parameters in the lifted space can cause
more computational burden in resulting SDP problems.
However, as pointed out in \cite{LP2010}, the {\itshape running intersection
property} holds true for these lifted parameters. Hence, like for polynomial
optimization problems \cite{LasserreSparsity,WKKMsparsity},
some sparse SDP relaxations for (\ref{eq::lsipp2}) can be explored to reduce
the computational cost.

\begin{example}\label{ex::appro}
Consider the one-sided $L_1$ approximation problem
\begin{equation}\label{eq::L1}
	\left\{\begin{aligned}
		\min_{x\in\RR^n}&\ \sum_{i=1}^n\frac{x_i}{i}\\
		\text{s.t.}&\ \sum_{i=1}^ny^{i-1}x_i-b(y)\ge 0,
		\quad \forall y\in [0,1].
	\end{aligned}\right.
\end{equation}
Here, we approximate two (semi-algebraic) functions \cite{GG} on $[0,1]$:
\[
	\text{\upshape(i):}\ \ b(y)=\frac{1}{2-y},\ n=8;\quad \text{and}\quad 
	\text{\upshape(ii):}\ \
	b(y)=-\frac{1}{1+y^2},\ n=10.
\]
Clearly, in order to convert this problem into LSIPP, 
we can add lifted variable $z$ such that $(2-y)z=1$ for case
(i) and $(1+y^2)z=-1$ for case (ii). Then, we
solve the SDP relaxations \eqref{eq::psdplsipp} with order $k=4$ for
(i) and $k=5$ for (ii) by YALMIP. The obtained coefficients
$x_i$'s are listed below
\[
	\begin{aligned}
		\text{(i):}\
		&(0.5000,0.2501,0.1227,0.0787,-0.0258,0.1226,-0.0967,0.0484),\\
		\text{(ii):}\
		&(-1.0000,-0.0000,1.0016,-0.0202,-0.8566,-0.6123,2.6222,-2.6059,\\
		&1.1881,-0.2168)
	\end{aligned}
\]
We show the accuracy of the computed optimal approximations
(denoted by $f$) of $b(Y)$ in Figure \ref{fig::approx}.
\begin{figure}
\centering
\includegraphics[width=0.55\textwidth]{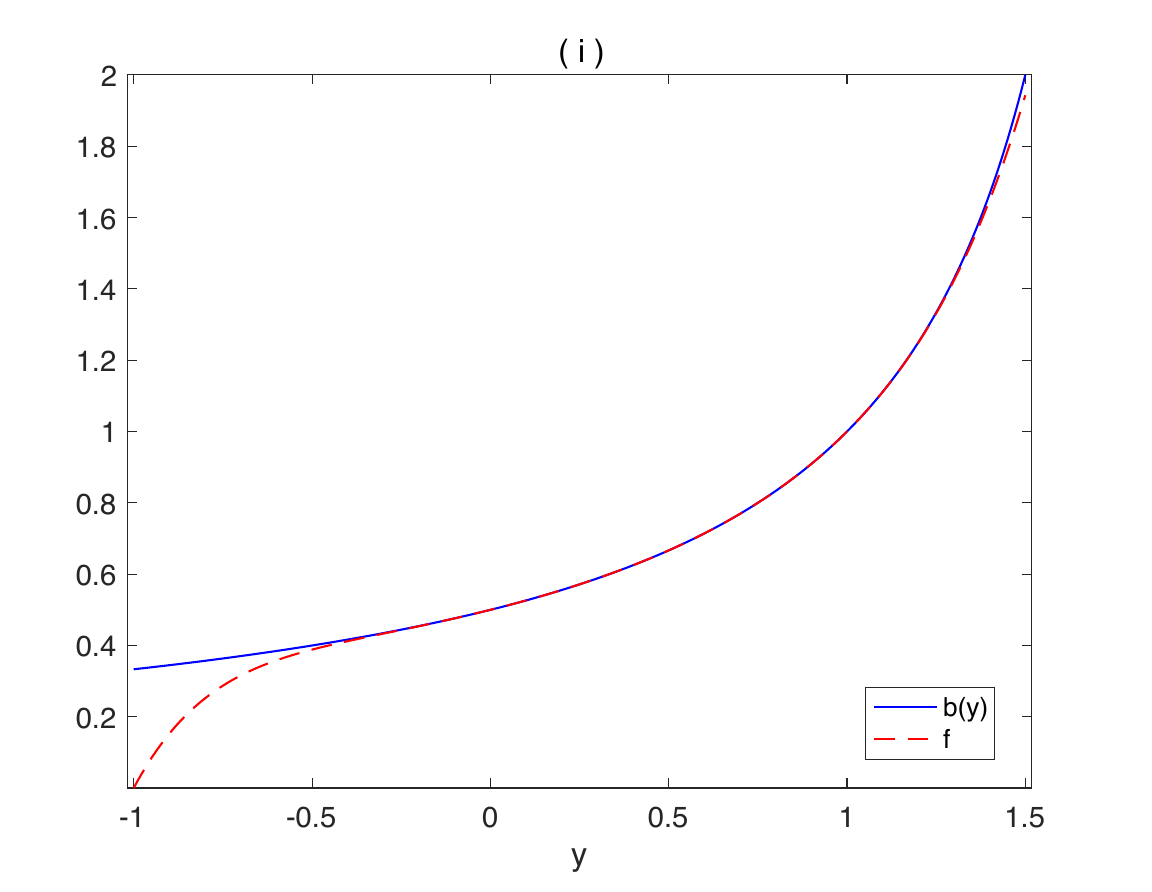}\quad
\includegraphics[width=0.35\textwidth]{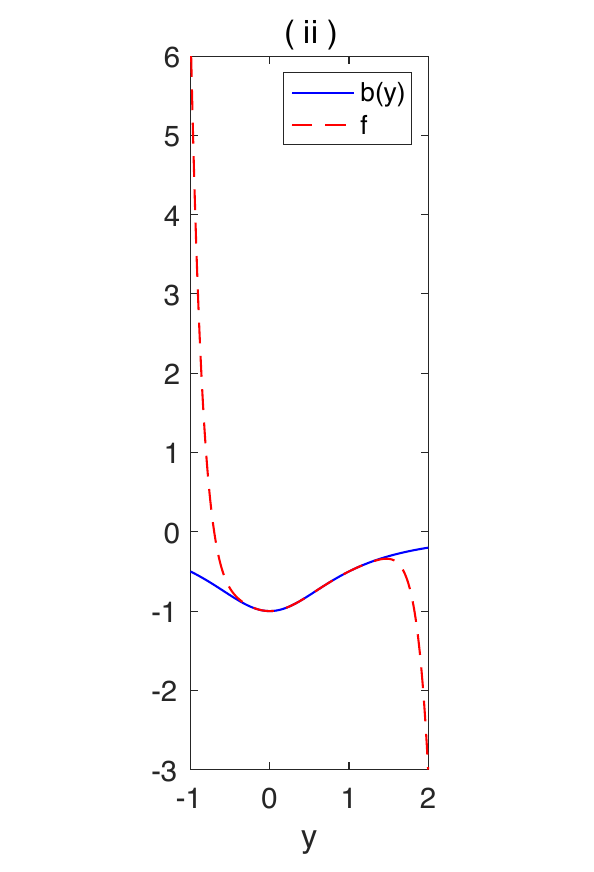}\\
\caption{\label{fig::approx}
Pictures for the univariate approximation problems in Example
\ref{ex::appro}.}
\end{figure}
\end{example}

\begin{example}\label{ex::sa2}
In $(x_1,x_2)$-plane, consider the intersection area $\mathcal{F}$  of $x_2\ge 0$,
$1-x_1^2-x_2^2\ge 0$ and $x_1+1-x_2^2\ge 0$. Then, $\mathcal{F}$ can
also be seen as the interstion of $x_2\ge 0$, the half planes defined
by the lines tangent to $1-x_1^2-x_2^2=0$ in the first quadrant and  
to $x_1+1-x_2^2=0$ in the second quadrant, as shown in Figure \ref{fig::sa2}. 
Therefore, it is easy to check that
\[
	\mathcal{F}=\{(x_1,x_2)\in\RR^2\mid a(y_1)^Tx+b(y_1)\ge 0,\
	\forall\ y_1\in[-1,1]\}\cap\mathcal{X},
\]
where $\mathcal{X}=\{(x_1,x_2)\in\RR^2\mid x_2\ge 0\}$, 
\[
	\begin{aligned}
		&a_1(y_1)=-\min[y_1,0]-\max[y_1,0]y_1,\\
		&a_2(y_1)=2\min[y_1,0]\sqrt{y_1+1}-\max[y_1,0]\sqrt{1-y_1^2},\\
		&b(y_1)=-\min[y_1,0](2+y_1)+\max[y_1,0].
	\end{aligned}
\]
Here, the equations $a(y_1)^TX+b(y_1)=0$ for $y_1\in[-1,1]$, in fact, represent the
tangent lines mentioned above.
Consider the LSIP problem 
\[
\left\{
\begin{aligned}
	p^*:=\min_{x\in\mathcal{X}}&\ c^Tx\\
	\text{s.t.}&\ a(y_1)^Tx+b(y_1)\ge 0,\ \ \forall y\in[-1,1],
\end{aligned}\right.
\]
in two cases: (i) $c=(1,-1)$; (ii) $c=(-1,-1)$. We can verify that the
minima and minimizers are: (i) $p^*=-\frac{5}{4}$,
$x^*=(-\frac{3}{4},\frac{1}{2})$; (ii) $p^*=-\sqrt{2}$,
$x^*=(\frac{\sqrt{2}}{2},\frac{\sqrt{2}}{2})$. Now, we first convert
this LSIP problem into LSIPP by the lifting method and then solve it by the
SDP relaxations \eqref{eq::psdplsipp}. As $2\cdot\min[y_1,0]=y_1-|y_1|$ and
$2\cdot\max[y_1,0]=y_1+|y_1|$, we can write
\[
	\mathcal{F}=\{(x_1,x_2)\in\RR^2\mid \td{a}(y)^Tx+\td{b}(y)\ge 0,\
	\forall\ y\in S\}\cap\mathcal{X},
\]
where 
\[
	\begin{aligned}
		\td{a}_1(y)&=-y_1+y_2-y_1^2-y_1y_2,\\
		\td{a}_2(y)&=2y_1y_3-2y_2y_3-y_1y_4-y_2y_4,\\
		\td{b}(y)&=-(2+y_1)(y_1-y_2)+y_1+y_2, 
	\end{aligned}
\]
and 
\[
	S=\{y\in\RR^4\mid -1\le y_1\le 1, y_2^2=y_1^2, y_2\ge 0,
	1+y_1=y_3^2, y_3\ge 0, 1-y_1^2=y_4^2, y_4\ge 0\}. 
\]
Now we can use the SDP relaxations \eqref{eq::psdplsipp} to solve the
obtained LSIPP problems. The approximate minima and minimizers are:
(i) $p_5^{\sos}=-1.2496$, $\td{x}^{(5)}=(-0.7575,0.4921)$; (ii)
$p_5^{\sos}=-1.3936$, $\td{x}^{(5)}=(0.6544,0.7392)$.
\begin{figure}
\centering
\includegraphics[width=0.55\textwidth]{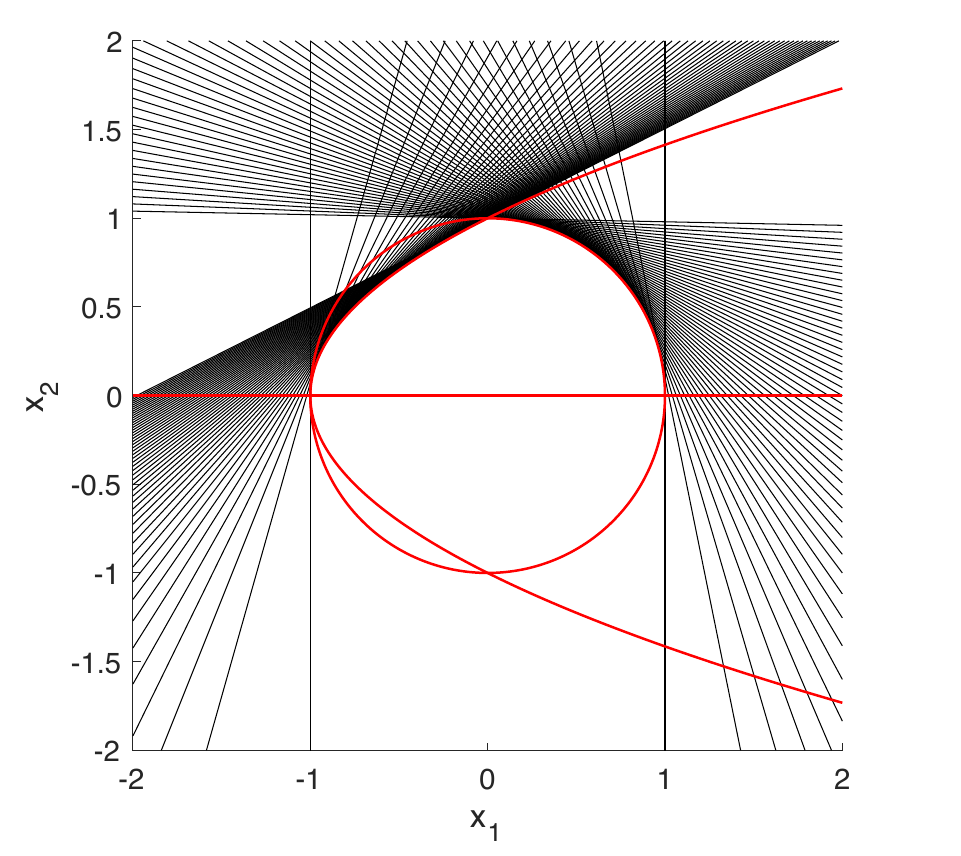}
\caption{\label{fig::sa2}
The feasible set $\mathcal{F}$ in Example \ref{ex::sa2}.}
\end{figure}
\end{example}

\subsection{S.O.S-convex objectives}
Next, by the exact SDP relaxations
for classes of nonlinear SDP problems proposed in \cite{JL2012}, we
extend the SDP relaxation method in Section~\ref{section::SDPrelax} to
the following semi-infinite programming problem
\begin{equation}\label{eq::sosconvexsip}
\left\{
\begin{aligned}
	h^*:=\inf_{x\in\RR^m}&\ h(x)\\
	\text{s.t.}&\ a(y)^Tx+b(y)\ge 0,\ \ \forall y\in S,
\end{aligned}\right.
\end{equation}
where $a(Y)$, $b(Y)$, $S$ are defined as in \eqref{eq::lsipp} and
$h(X)\in\RR[X]$ is s.o.s-convex polynomial. Recall that 
\begin{definition}{\upshape(\cite{SRCSHN})}
	A polynomial $h\in\RR[Y]$ is s.o.s-convex if its Hessian $\nabla^2
	h$ is a s.o.s, i.e., there are an integer $r$ and a matrix
	polynomial $H\in\RR[Y]^{r\times n}$ such that $\nabla^2
	h(Y)=H(Y)^TH(Y)$.
\end{definition}
While checking the convexity of a polynomial is generally
NP-hard \cite{Ahmadi2013},
s.o.s-convexity can be checked numerically by solving an
SDP, see \cite{SRCSHN}. 

We first relax \eqref{eq::sosconvexsip} with arbitrary convex polynomial
objective function as 
\begin{equation}\label{eq::sosconvexsip2}
	h^{\sos}:=\inf_{x\in\RR^m}\ \ h(x)\quad\text{s.t.}\
	a(Y)^Tx+b(Y)\in\Qm(G).
\end{equation}
\begin{theorem}\label{th::equi2}
	If $h(X)$ is convex, $\Qm(G)$ is Archimedean and
	the Slater condition holds for \eqref{eq::sosconvexsip}, then
	$h^{\sos}=h^*$.
\end{theorem}
\begin{proof}
	As $h(X)$ is convex, replacing the function $c^TX$ in the proof of
	Theorem~\ref{th::equi} by $h(X)$, all arguments in the proof
	of Theorem~\ref{th::equi} are still valid. In particular,
	\eqref{eq::inequality} become
	\[
		\begin{aligned}
			p^{\sos}-p^*&\le h(\hat{x})-p^*\\
			&\le(1-\delta)h(x')+\delta h(\bar{x})-p^*\\
			&=(h(x')-p^*)+\delta(h(\bar{x})-h(x'))\\
			&<\frac{\varepsilon}{2}+\frac{\varepsilon}{2}
			=\varepsilon,
		\end{aligned}
	\]
	due to the convexity of $h$. 
	Then the conclusion follows.
\end{proof}
Recall that $\mathcal{F}_k$ denotes the feasible set of
\eqref{eq::psdplsipp}.  For $k\ge d_P$,
replacing $\Qm(G)$ in $(\ref{eq::sosconvexsip2})$ by its $k$-th
truncation $\Qm_k(G)$, we get
\begin{equation}\label{eq::sosconvexsdp}
	h_k^{\sos}:=\inf\ \ h(x)\quad\text{s.t.}\
	x\in\mathcal{F}_k. 
\end{equation}
Consequently, it follows from the proof of Theorem
\ref{th::convergence} that 
\begin{theorem}\label{th::convergence2}
	If $h(X)$ is convex, $\Qm(G)$ is Archimedean and
	the Slater condition holds for \eqref{eq::sosconvexsip}, then
	$h^{\sos}_k$ decreasingly converges to $h^*$ as
	$k\rightarrow\infty$.
\end{theorem}
Moreover, if $h(X)$ is s.o.s-convex, we point out that for each $k\ge d_P$, 
\eqref{eq::sosconvexsdp} is equivalent to a {\itshape single} SDP
problem under certain conditions as shown in \cite{JL2012}. In fact,
it is easy to see that there exist some integers
$l, t$ and real $t\times t$ symmetric matrices $\{A_i\}_{i=0}^m$ and
$\{B_j\}_{j=1}^l$ such that $\mathcal{F}_k$ is identical with 
\[
	\left\{x\in\RR^m\ \Big|\ \exists w\in\RR^l,\ \text{s.t.}\ A_0+\sum_{i=1}^m
	A_ix_i+\sum_{j=1}^l B_jw_j\succeq 0\right\},
\]
where $w_j$'s correspond to the entries of $Z_j$'s in
\eqref{eq::SDPform}. Thus, \eqref{eq::sosconvexsdp} becomes
\begin{equation}\label{eq::sosconvexsdp2}
	\left\{
	\begin{aligned}
		h_k^{\sos}=\inf_{x\in\RR^m, w\in\RR^l}&\ \ h(x)\\
		\text{s.t.}&\ \  A_0+\sum_{i=1}^m A_ix_i+\sum_{j=1}^l
		B_jw_j\succeq 0. 
	\end{aligned}\right.
\end{equation}
Let $d\in\N$ is the {\itshape smallest} even number such that $d\ge\deg{h}$. 
Denote the variables $W=(W_1,\ldots,W_l)$ and let $\Sigma^2_d[X,W]$ be
the set of sums of squares of polynomials in $\RR[X,W]$ of degree up
to $d$.   
Consider the dual problem of \eqref{eq::sosconvexsdp2}
\begin{equation}\label{eq::sosconvexsdpd}
	\left\{
	\begin{aligned}
		\lambda_k:=\sup_{\lambda\in\RR, V\in\mathbb{S}_+^m}&\ \lambda\\
		\text{s.t.}&\ h(X)-\sum_{i=1}^m\langle A_i, V\rangle\cdot X_i-\sum_{j=1}^l\langle B_j,
		V\rangle\cdot W_j-\langle A_0,
		V\rangle-\lambda\\
		&\ \in\Sigma^2_d[X,W],
	\end{aligned}\right.
\end{equation}
which can be reduced to an SDP problem as shown in
Section~\ref{section::SDPrelax}. Clearly,  $h^{\sos}_k\ge \lambda_k$. 
We have $h^{\sos}_k=\lambda_k$ under
certain conditions (c.f. \cite[Theorem~3.1]{JL2012}).
Therefore, an SDP relaxation method
is drived for \eqref{eq::sosconvexsip} with s.o.s-convex objectives. 

\begin{example}\cite{Coope1985}
Consider the followin SIP problem
\[
\left\{
\begin{aligned}
	h^*=\inf_{x\in\RR^3}&\ x_1^2+x_2^2+x_3^2\\
	\text{s.t.}&\
	-(y_1+y_2^2+1)x_1-(y_1y_2-y_2^2)x_2-(y_1y_2+y_2^2+y_2)x_3-1\ge
	0,\\
	&\ \forall y\in [0,1]^2.
\end{aligned}\right.
\]
The exact minimum and minimizer are $h^*=1$ and $x^*=(-1,0,0)$
\cite{Coope1985}. Clearly, the objective function is s.o.s-convex.
Solving the SDP problem
\eqref{eq::sosconvexsdpd}, 
we obtain $\lambda_1=1.0000$ with approximate minimizer $(-1.0000,
2.3149\times 10^{-6},-1.0410\times 10^{-5})$.
\end{example}

\begin{example}
Consider the following SIP problem
\[
\left\{
\begin{aligned}
	h^*=\inf_{x\in\RR^2}&\
	h(x)=\frac{5}{8}x_1^2-4x_1+\frac{3}{4}x_1x_2-4x_2+\frac{5}{8}x_2^2+8\\
	\text{s.t.}&\
	1-y_1x_1-y_2x_2\ge 0,\ \forall y\in\{y\in\RR^2\mid y_1^2+y_2^2=1\}.
\end{aligned}\right.
\]
Clearly, $h(X)$ is s.o.s-convex.
It is easy to see that the feasible set is the closed unit disk around
the origin. If we let $f(X)=(X_1-2)^2+\frac{(X_2-2)^2}{4}$, then
$h(X)=f(\frac{(X_1-2)+(X_2-2)}{\sqrt{2}}+2,\frac{-(X_1-2)+(X_2-2)}{\sqrt{2}}+2)$.
Geometrically, for any $r>0$, the curve $h(X)=r$ can be obtained by rotating
the ellipse $f(X)=r$ around $(2,2)$ by
$45^{\circ}$ counterclockwise. Therefore, the minimizer is
$x^*=(\frac{\sqrt{2}}{2},\frac{\sqrt{2}}{2})$ with the minimum
$h^*=9-4\sqrt{2}\approx 3.3431$. 
Solving the SDP problem \eqref{eq::sosconvexsdpd},
we obtain $\lambda_1=3.3432$ with approximate minimizer $(0.7071,0.7071)$. 
\end{example}
\section{Conclusion}\label{sec::conclusion}
In this paper, a hierarchy of SDP relaxations for LSIPP problems is
presented. It can be seen as the dual of Lasserre's relaxations for GPM
problems and enjoys several desirable features. 
Some (approximate) minimizers of LSIPP problems can be
extracted using these SDP relaxations, which is useful in some
applications. Convergence rate
of these SDP relaxations is estimated using some existing results.
We also extend this SDP relaxation method to more general semi-infinite
programming problems.
%
%
%
%
%

\vspace{10pt}
\noindent{\bfseries Acknowledgments:}
The authors are very grateful for the comments of two anonymous
referees which helped to improve the presentation. 
The authors would like to thank Miguel A. Goberna for helpful comments
on the Lipschitz continuity of the optimal value function of
\eqref{eq::lsipp} used in Lemma~\ref{lem::perturbation}. 
Feng Guo is supported by
the Chinese National Natural Science Foundation under grants 11401074, 11571350.
Xiaoxia Sun is supported by the Chinese National Natural
Science Foundation under grant 11801064, the
Foundation of Liaoning Education Committee under grant LN2017QN043.

\end{document}